\documentclass[12pt]{amsart}
\usepackage[T1]{fontenc}
\usepackage[finnish, english]{babel}
\usepackage[utf8]{inputenc}

\usepackage{amsmath}
\usepackage{amsfonts}
\usepackage{amssymb}
\usepackage{amsthm}
\usepackage{url}

\newcommand{\R}{\mathbb{R}}
\newcommand{\dif}[0]{\ensuremath{\,\mathrm{d}}}
\newcommand{\norm}[1]{\ensuremath{\Vert #1 \Vert}}

\newcommand{\abs}[1]{\ensuremath{\vert #1 \vert}}
\newcommand{\scabs}[1]{\ensuremath{\left\vert #1 \right\vert}}

\newcommand{\B}[0]{\ensuremath{\mathcal{B}}}

\newcommand{\om}{\ensuremath{m^{\sharp}}}
\newcommand{\um}{\ensuremath{m^{\flat}}}

\DeclareMathOperator*{\esssup}{ess\,sup}

\DeclareMathOperator*{\spt}{supp}

\DeclareMathOperator*{\dive}{div}

\def\vint_#1{\mathchoice%
          {\mathop{\kern 0.2em\vrule width 0.6em height 0.69678ex depth -0.58065ex
                  \kern -0.8em \intop}\nolimits_{\kern -0.4em#1}}%
          {\mathop{\kern 0.1em\vrule width 0.5em height 0.69678ex depth -0.60387ex
                  \kern -0.6em \intop}\nolimits_{#1}}%
          {\mathop{\kern 0.1em\vrule width 0.5em height 0.69678ex depth -0.60387ex
                  \kern -0.6em \intop}\nolimits_{#1}}%
          {\mathop{\kern 0.1em\vrule width 0.5em height 0.69678ex depth -0.60387ex
                  \kern -0.6em \intop}\nolimits_{#1}}}

\theoremstyle{plain}
\newtheorem{theorem}{Theorem}
\newtheorem{lemma}[theorem]{Lemma}
\newtheorem{proposition}[theorem]{Proposition}

\numberwithin{theorem}{section}
\numberwithin{equation}{section}

\theoremstyle{definition}
\newtheorem{definition}[theorem]{Definition}

\theoremstyle{remark}
\newtheorem{remark}[theorem]{Remark}

\author{Teemu Lukkari}

\address[Teemu Lukkari]{Department of Mathematics and Statistics\\
  P.O. Box 35 (MaD)\\ 40014 University of
  Jyv\"askyl\"a\\Jyv\"askyl\"a, Finland} \email{teemu.j.lukkari@jyu.fi}
\date{\today}



\begin{document}

\title{Stability of solutions to nonlinear diffusion equations}

\begin{abstract}
  We prove stability results for nonlinear diffusion equations of the
  porous medium and fast diffusion types with respect to the
  nonlinearity power $m$: solutions with fixed data converge in a
  suitable sense to the solution of the limit problem with the same
  data as $m$ varies. Our arguments are elementary and based on a
  general principle. We use neither regularity theory nor nonlinear
  semigroups, and our approach applies to e.g. Dirichlet problems in
  bounded domains and Cauchy problems on the whole space.
\end{abstract}

\maketitle

\section{Introduction}

We study the stability of positive solutions to the parabolic equation
\begin{equation}
  \label{eq:eq-intro}
  \partial_t u-\Delta u^m=0
\end{equation}
with respect to perturbations in the nonlinearity power $m$. The main
issue we address is whether solutions with fixed boundary and initial
data converge in some sense to the solution of the limit problem as
$m$ varies.  This kind of stability questions are not only a matter of
merely mathematical interest; in applications, parameters like $m$ are
often known only approximately, for instance from experiments.  Thus
it is natural to ask whether solutions are sensitive to small
variations in such parameters or not.

The equation \eqref{eq:eq-intro} is an important prototype of a
nonlinear diffusion equation.  For $m>1$, this is called the porous
medium equation (PME), and $m<1$, the fast diffusion equation (FDE).
The PME is degenerate, the diffusion being slow when $u$ is small.
The FDE is singular, and the opposite happens: the diffusion is fast
when $u$ is small. We do not exclude the case $m=1$, when we have the
ordinary heat equation. However, we do restrict our attention to the
supercritical range:
\begin{displaymath}
  m>m_c,\quad \text{where}\quad m_c=(n-2)_+/n.
\end{displaymath}
For the basic theory of the porous medium and fast diffusion
equations, we refer to the monographs \cite{DaskalopoulosKenig,
  VazquezBook2, VazquezBook} and the references therein.

It turns out that stability with respect to $m$ for
\eqref{eq:eq-intro} with convergence in an $L^p$ space can be
established in a relatively elementary manner. The reason is that weak
solutions to \eqref{eq:eq-intro} are defined in terms of the function
$u^m$ instead of $u$. This means that $u^m$ and its gradient are
always $L^2$ functions, even if $m$ varies.  The starting point of our
argument is a compactness property of weak solutions: locally
uniformly bounded sequences of weak solutions contain pointwise almost
everywhere convergent subsequences.  With the compactness result in
hand, stability for a particular problem follows by verifying the
local uniform boundedness and that the correct initial or boundary
values are attained. Only some fairly simple estimates are needed for
the second step.  More spesifically, we do not need H\"older
continuity, Harnack's inequality, or reverse H\"older inequalities for
the gradient.  We carry out the latter step in detail in two cases:
for Dirichlet problems with nonzero boundary values on bounded
domains, and for Cauchy problems on the whole space, with initial data
a positive measure of finite mass.

The previous result closest to ours is that of B\'enilan and
Crandall~\cite{BenilanCrandall}. They prove the stability of mild
solutions to
\begin{equation}\label{eq:general-phi}
  \partial_t u-\Delta \varphi(u)=0
\end{equation}
with respect to $\varphi$ using the theory of nonlinear
semigroups. Here $\varphi$ can be a maximal monotone graph; in the
case of a power function, the assumption used in
\cite{BenilanCrandall} reduces to $m\geq m_c$.  Our approach is
different from that of \cite{BenilanCrandall}, as we do not employ the
machinery of nonlinear semigroups.  

Explicit estimates for the difference of two solutions to the Cauchy
problem can be established by employing Kru\v{z}kov's ``doubling of
variables'' technique, see \cite{ChenKarlsen, CG, KarlsenRisebro}.
However, these estimates hold only under some restrictions; the
initial values need to be bounded, and these results apply only in the
degenerate case $m\geq 1$. See also \cite{PanApprox, PanGang} for some
estimates in the onedimensional situation.

Our result on the Cauchy problem applies with initial data a measure
with finite mass. An example of such a situation is provided by the
celebrated Barenblatt solutions \cite{Barenblatt, ZeldovichKompaneets}.
For $m>1$, it is given by
\begin{equation}\label{eq:barenblatt-pme}
  \B_m(x,t)=
  \begin{cases}
    t^{-\lambda}\left(C-\frac{\lambda(m-1)}{2mn}
      \frac{\abs{x}^2}{t^{2\lambda/n}}\right)_+^{1/(m-1)}, & t>0,\\
    0, & t\leq 0,
  \end{cases}
\end{equation}
where
\begin{displaymath}
  \lambda=\frac{n}{n(m-1)+2}.
\end{displaymath}
For $m_c<m<1$, it is convenient to write the formula as
\begin{equation}\label{eq:barenblatt-fde}
  \B_m(x,t)=
  \begin{cases}
    t^{-\lambda}\left(C+k
      \frac{\abs{x}^2}{t^{2\lambda/n}}\right)^{-1/(1-m)}, & t>0,\\
    0, & t\leq 0,
  \end{cases}
\end{equation}
where
\begin{displaymath}
  k=\frac{\lambda(1-m)}{2mn}.
\end{displaymath}
Note that $\lambda$ and $k$ are strictly positive, since here
$m_c<m<1$.  If one chooses the constant $C$ so that the normalization
\begin{displaymath}
  \int_{\Omega}\B_m(x,t)\dif x=1
\end{displaymath}
holds for all $t>0$, the initial trace of $\B_m$ is the Dirac measure
at the origin. Hence our results imply that
\begin{displaymath}
  \B_{m_i}\to \B_m \quad\text{ as }\quad m_i\to m
\end{displaymath}
in an $L^p$ space. As far as we know, the previous stability results
do not contain this fact.

Our results include the case when the limiting problem is the heat
equation. In this situation, we do not need the restrictions $m\geq 1$
or $m\leq 1$ on the approximating problems; both degenerate ($m>1$)
and singular ($m<1$) problems are allowed.  However, the restriction
$m>m_c$ seems essential, as a number of the tools we use are known to
fail when $0<m\leq m_c$. For instance, local weak solutions might no
longer be locally bounded, and the $L^1$-$L^\infty$ smoothing effect
for the Cauchy problem fails. See \cite{CF1, Sacks} for case when
$m\to \infty$, and \cite{DGL, Hui, Hui2} for the case $m\to 0$.

Generalizing our argument to other boundary and initial conditions is
straightforward. For instance, Neumann boundary conditions and Cauchy
problems with growing initial data can be handled
similarly. Generalization to nonlinearities other than powers, as in
\eqref{eq:general-phi}, should also be possible, albeit less
straightforward.

We also provide an alternative argument for the stability of Dirichlet
problems in bounded domains in the case $m\geq 1$. Compared to the
previous argument, the advantages of this approach are the fact that
local boundedness is not needed, and that one can in addition estimate
the difference of two solutions in terms of the difference of the
respective nonlinearity powers. The disadvantages are the restriction
$m\geq 1$ and the fact that the proof seems less amenable to
generalizations, since we employ strong monotonicity. See
\cite{PanApprox, PanGang} for similar estimates in the onedimensional
case.

Our proofs illustrate the differences between the the $p$-parabolic
equation
\begin{equation}\label{eq:p-para}
  \partial_t u-\dive(\abs{\nabla u}^{p-2}\nabla u)=0
\end{equation}
and the equation \eqref{eq:eq-intro}.  For \eqref{eq:p-para}, changing
the exponent $p$ also changes the space to which weak solutions
belong. Dealing with this in the case of \eqref{eq:p-para} is quite
delicate, and one needs the fact that the gradient satisfies a reverse
H\"older inequality. The equation \eqref{eq:eq-intro} has the
remarkable advantage that we may work in a fixed space, and our
results do not rely on the more sophisticated tools of regularity
theory.  The case of \eqref{eq:p-para} is detailed in
\cite{ParaStab}. See also \cite{LqvistStab} equations similar to the
$p$-Laplacian, \cite{LiMartio} for obstacle problems,
\cite{LqvistRayleigh} for eigenvalue problems, and \cite{ABKO} for
triply nonlinear equations.

The paper is organized as follows. In Section \ref{sec:weak}, we
recall the necessary background material, in particular the definition
of weak solutions. Section \ref{sec:local-stability} contains the
proof of the compactness theorem for locally bounded sequences of weak
solutions. The actual stability results are then established in
Sections \ref{sec:stab-dirichlet} and \ref{sec:stab-cauchy}, for
Dirichlet problems in the former and Cauchy problems in the latter.
We finish by presenting the alternative proof for Dirichlet problems
in Section \ref{sec:alt-stab}.

\section{Weak solutions}
\label{sec:weak}

Let $\Omega$ be an open subset of $\R^n$, and let $0<t_1<t_2<T$. We
use the notation $\Omega_T=\Omega\times(0,T)$ and $U_{t_1,t_2}=U\times
(t_1,t_2)$, where $U\subset\Omega$ is open. The parabolic boundary
$\partial_p U_{t_1,t_2}$ of a space-time cylinder $U_{t_1,t_2}$
consists of the initial and lateral boundaries, i.e.
\begin{displaymath}
  \partial_p U_{t_1,t_2}=(\overline{U}\times\{t_1\})
  \cup (\partial U\times [t_1,t_2]).
\end{displaymath}
The notation $U_{t_1,t_2}\Subset\Omega_T$ means that the closure
$\overline{U_{t_1,t_2}}$ is compact and
$\overline{U_{t_1,t_2}}\subset\Omega_T$.

We use $H^1(\Omega)$ to denote the usual Sobolev space, the space of
functions $u$ in $L^2(\Omega)$ such that the weak gradient exists and
also belongs to $L^2(\Omega)$. The norm of $H^1(\Omega)$ is
\begin{displaymath}
  \norm{u}_{H^1(\Omega)}=\norm{u}_{L^2(\Omega)}+\norm{\nabla u}_{L^2(\Omega)}.
\end{displaymath}
The Sobolev space with zero boundary values, denoted by
$H^{1}_0(\Omega)$, is the completion of $C^{\infty}_0(\Omega)$ with
respect to the norm of $H^1(\Omega)$. The dual of $H^1_0(\Omega)$ is
denoted by $H^{-1}(\Omega)$.

The parabolic Sobolev space $L^2(0,T;H^1(\Omega))$ consists of
measurable functions $u:\Omega_T\to[-\infty,\infty]$ such that
$x\mapsto u(x,t)$ belongs to $H^1(\Omega)$ for almost all
$t\in(0,T)$, and 
\begin{displaymath}
  \int_{\Omega_T}\abs{u}^2+\abs{\nabla u}^2\dif x\dif t<\infty.
\end{displaymath}
The definition of $L^2(0,T;H^{1}_0(\Omega))$ is identical, apart from
the requirement that $x\mapsto u(x,t)$ belongs to $H^{1}_0(\Omega)$.
We say that $u$ belongs to $L^2_{loc}(0,T;H^{1}_{loc}(\Omega))$ if
$u\in L^2(t_1,t_2;H^1(U))$ for all $U_{t_1,t_2}\Subset\Omega_T$.

We use the following Sobolev inequality. See \cite[Proposition 3.1,
p. 7]{DiBenedettoBook} for the proof.
\begin{lemma}\label{lem:sobolev}
  Let $u$ be a function in $L^2(0,T;H^{1}_0(\Omega))$.
  Then we have
  \begin{equation}\label{eq:sobolev-further}
    \int_{\Omega_T} \abs{u}^{2\kappa }\dif x\dif t\leq 
    C\int_{\Omega_T}\abs{\nabla u}^2\dif x\dif
    t\left(\esssup_{0<t<T}
      \int_\Omega u^{1+1/m}\dif x\right)^{2/n},
  \end{equation}
  where 
  \begin{equation}\label{eq:kappa}
    \kappa=1+\frac{1}{n}+\frac{1}{mn}.
  \end{equation}
\end{lemma}  

Solutions are defined in the weak sense in the parabolic Sobolev
space.
\begin{definition}\label{def:local-weak}
  Assume that $m> m_c$.  A nonnegative function $u:\Omega_T\to\R$ is a
  local weak solution of the equation
  \begin{equation}\label{eq:pme}
    \frac{\partial u}{\partial t}-\Delta u^m=0
  \end{equation}
  in $\Omega_T$, if $u^m\in
  L^2_{loc}(0,T;H^{1}_{loc}(\Omega))$ and
  \begin{equation}\label{eq:weak-pme}
    \int_{\Omega_T}-u\frac{\partial\varphi}{\partial t}
    +\nabla u^m\cdot\nabla\varphi\dif x\dif t=0
  \end{equation}
  for all smooth test functions $\varphi$ compactly supported in
  $\Omega_T$.   
\end{definition}
We will always assume that $m> m_c$, and consider only nonnegative
solutions.  We refer to the monographs \cite{DaskalopoulosKenig,
  VazquezBook2, VazquezBook} for the basic theory related to this type
of equations, and numerous further references.  In particular, weak
solutions have a locally H\"older continuous representative, see
\cite{DahlbergKenig,  DiBenedettoFriedman} or Chapter 7 of
\cite{VazquezBook}; however, we do not use this fact.

An important example of a local weak solution in the sense of
Definition \ref{def:local-weak} is provided by the celebrated
Barenblatt solutions \cite{Barenblatt, ZeldovichKompaneets}, given by
\eqref{eq:barenblatt-pme} and \eqref{eq:barenblatt-fde}.  Observe that
for $m>1$, $\B_m$ has compact support in space at each time instant,
while for $m<1$, $\B_m$ has a powerlike tail. In both cases, $\B_m$ is
a local weak solution in $\R^n\times(0,\infty)$.  See Section
\ref{sec:stab-cauchy} for the trace of $\B_m$ at the initial time
$t=0$.

The definition of weak solutions and supersolutions does not include a
time derivative of $u$. However, we would like to use test functions
depending on $u$, and thus the time derivative $\frac{\partial
  u}{\partial t} $ inevitably appears. To deal with this defect, a
mollification procedure in the time direction, for instance Steklov
averages or convolution with the standard mollifier, is usually
employed. The mollification
\begin{equation}\label{eq:naumann-conv}
  u^\ast(x,t)=\frac{1}{\sigma}\int_0^t e^{(s-t)/\sigma}u(x,s)\dif s
\end{equation}
is convenient.  The aim is to obtain estimates independent of the time
derivative of $u^\ast$, and then pass to the limit $\sigma\to 0$.

The basic properties of the mollification \eqref{eq:naumann-conv} are
given in the following lemma, see \cite{Naumann}.
\begin{lemma}\label{lem:conv-prop}
  \begin{enumerate}
  \item If $u\in L^p(\Omega_T)$, then
    \begin{displaymath}
      \norm{u^\ast}_{L^p(\Omega_T)}\leq\norm{u}_{L^p(\Omega_T)},
    \end{displaymath}
    \begin{equation}\label{eq:naumann-timederiv}
      \frac{\partial u^\ast}{\partial t}=\frac{u-u^\ast}{\sigma},
    \end{equation}
    and $u^\ast\to u$ in $L^p(\Omega_T)$ as $\sigma\to 0$.
  \item If $\nabla u\in L^p(\Omega_T)$, then $\nabla(u^\ast)=(\nabla
    u)^\ast$,
    \begin{displaymath}
      \norm{\nabla u^\ast}_{L^p(\Omega_T)}\leq\norm{\nabla u}_{L^p(\Omega_T)},
    \end{displaymath}
    and $\nabla u^\ast\to \nabla u$ in $L^p(\Omega_T)$ as $\sigma \to
    0$.
  \item If $u_k\to u$ in $L^p(\Omega_T)$, then also
    \begin{displaymath}
      u^\ast_k\to u^\ast \text{ and }\frac{\partial u^\ast_k}{\partial
        t}\to \frac{\partial u^\ast}{\partial  t}
    \end{displaymath}
    in $L^p(\Omega_T)$.
  \item If $\nabla u_k\to \nabla u$ in $L^p(\Omega_T)$, then $\nabla
    u^\ast_k\to \nabla u^\ast$ in $L^p(\Omega_T)$.

  \item Similar results hold for weak convergence in $L^p(\Omega_T)$.

  \item If $\varphi\in C(\overline{\Omega_T})$, then
    \begin{displaymath}
      \varphi^\ast(x,t)+e^{-t/\sigma}\varphi(x,0)\to \varphi(x,t)
    \end{displaymath}
    uniformly in $\Omega_T$ as $\sigma \to 0$.
  \end{enumerate}
\end{lemma}

We use the following estimate for the local version of our stability
result. See \cite[Lemma 2.15]{KinnunenLindqvist2} or \cite[Lemma
2.9]{FDE-meas} for the proof.

\begin{lemma}\label{lem:caccioppoli}
  Let $u$ be a weak solution such that $0\leq u\leq M<\infty$, where
  $M\geq 1$, and let $\eta$ be any nonnegative function in
  $C^{\infty}_0(\Omega)$. Then
  \begin{displaymath}
    \int_{\Omega_T}\eta^2\abs{\nabla u^m}^2\dif x\dif t\leq
    2M^{m+1}\int_\Omega\eta^2\dif x
    +16M^{2m}\int_{\Omega_T}\abs{\nabla \eta}^2\dif x\dif t.
  \end{displaymath}
\end{lemma}

We use the following elementary lemma to pass pointwise convergences
between various powers.
\begin{lemma}\label{lem:conv-of-powers}
  Let $(f_i)$ be a sequence of positive functions on a measurable set
  $E$ with finite measure such that
  \begin{displaymath}
    f_i\to f \quad \text{in }L^1(E)\text{and pointwise almost everywhere}.
  \end{displaymath}
  Assume that $\alpha_i\to \alpha$ as $i\to \infty$. Then
  \begin{displaymath}
    f^{\alpha_i}_i\to f^\alpha
  \end{displaymath}
  pointwise almost everywhere.
\end{lemma}
\begin{proof}
  By Egorov's Theorem, for any $\varepsilon>0$ there is a set
  $F_\varepsilon$ such that $\abs{F_\varepsilon}<\varepsilon$ and
  $f_i\to f$ uniformly in $E\setminus F_\varepsilon$. Pick any point
  $x\in E\setminus F_\varepsilon$ such that 
  \begin{displaymath}
    0<\delta\leq f(x)\leq M<\infty.
  \end{displaymath}
  Then 
  \begin{displaymath}
    0<\delta/2\leq f_i(x)\leq 2M
  \end{displaymath}
  for all sufficiently large $i$ by uniform convergence. An
  application of the mean value theorem to the function $\alpha\mapsto
  t^\alpha$ gives
  \begin{align*}
    \abs{f_i(x)^{\alpha_i}-f(x)^\alpha}\leq & \abs{f_i(x)^{\alpha_i}-f_i(x)^\alpha}
    +\abs{f_i(x)^\alpha-f(x)^\alpha}\\
    \leq & c(\delta,M)\abs{\alpha_i-\alpha}+\abs{f_i(x)^\alpha-f(x)^\alpha}.
  \end{align*}
  By the convergence assumptions, it follows that
  $f_i(x)^{\alpha_i}\to f(x)^\alpha$ as $i\to\infty$.

  Since $\delta$, $M$, and $\varepsilon$ are arbitrary, the above
  implies that $f_i(x)^{\alpha_i}\to f(x)^\alpha$ for almost all $x$
  in the set $\{0<f(x)<\infty\}$. For points where $f(x)=0$, it is
  easy to check that $f_i(x)^{\alpha_i}\to 0$. Hence we have the
  desired convergence almost everywhere in the set $\{f(x)<\infty\}$,
  which is sufficient since $f$ is integrable.
\end{proof}

\section{Stability of local weak solutions}

\label{sec:local-stability}

In this section, we establish a local version of stability for a
bounded family of local weak solutions. This is beneficial since we
may then apply the same result to both initial--boundary value
problems in bounded domains and Cauchy problems on the whole space.
The crucial point in our stability results is extracting pointwise
convergent subsequences out of a sequence of solutions, in other
words, a compactness property of solutions.

Recall that we use the notation $U\Subset\Omega$ to mean that the
closure $\overline{U}$ is compact and contained in $\Omega$. In view
of applying this result to the Cauchy problem on $\R^n$, we allow the
cases $\Omega=\R^n$ and $T=\infty$.  We denote
\begin{displaymath}
  m^+=\sup_{i} m_i \quad \text{and}\quad m^-=\inf_{i} m_i,
\end{displaymath}
with similar notations for other exponents.

The main result of this section is the following theorem.
\begin{theorem}\label{thm:compactness}
  Let $m_i$, $i=1,2,3,\ldots,$ be exponents such that 
  \begin{equation}\label{eq:mconv}
    m_i\to m \quad \text{as} \quad i\to \infty 
    \quad \text{for some}\quad m>m_c=(n-2)_+/n
  \end{equation}
  Let $u_i$, $i=1,2,3,\ldots$, be positive local weak solutions to
  \begin{displaymath}
    \partial_t u_i-\Delta u_i^{m_i}=0
  \end{displaymath}
  in $\Omega_T$. Assume that we have the bound
  \begin{equation}\label{eq:assumptionM1}
    \norm{u_i}_{L^\infty(U_{t_1,t_2})}\leq M<\infty
  \end{equation}
  in all cylinders $U_{t_1,t_2}$ such that $U\Subset \Omega $ and
  $0<t_1 <t_2<T$.  
  
  Then there is a function $u$ such that $u_i\to u $ and $u^{m_i}\to
  u^m$ pointwise almost everywhere in $\Omega_T$, for a subsequence
  still indexed by $i$. Further, $u$ is a local weak solution to
  \begin{displaymath}
    \partial_t u-\Delta u^m=0.
  \end{displaymath}
\end{theorem}

We split the proof of Theorem \ref{thm:compactness} into several
lemmas.  In view of the convergence assumption \eqref{eq:mconv}, we
are free to assume that
\begin{displaymath}
  m_c< m^-\quad\text{and}\quad  m^+<\infty.
\end{displaymath}
Let us define the auxiliary exponents
\begin{displaymath}
  \om_i=\max\{m_i,1\}\quad \text{and}\quad \um_i=\min\{m_i,1\}.
\end{displaymath}
Then $\om_i\geq \um_i$, and one of these exponents always equals $m_i$
and the other equals one. We use a similar notation for the limit
exponent $m$.

The proof of Theorem \ref{thm:compactness} consists of three main
steps: first we apply a compactness result \cite{Simon} to certain
auxiliary functions to find the limit function $u$.  Then we show that
$u_i^{\om_i}$ converges to $u^{\om}$ in measure.  This is the most
involved part of the proof, and for a key estimate we apply a test
function due to Ole\u\i nik.  In the last step, we establish the
pointwise convergences $u_i\to u$ and $u_i^{m_i}\to u^m$ by applying
Lemma \ref{lem:conv-of-powers}, and the fact that $u$ is a local weak
solution by applying Lemma \ref{lem:caccioppoli}.

The next lemma will be used in proving the convergence in measure. 
\begin{lemma}\label{lem:diff-of-powers}
  Let $0\leq s<t\leq M$, where $M\geq 1$, be such that
  \begin{displaymath}
    t^{\om}-s^{\om}\geq \lambda >0.
  \end{displaymath}
  Then
  \begin{displaymath}
    t^{\um}-s^{\um}\geq \frac{\um}{\om}M^{\um-\om}\min\{\lambda,\lambda^{\um/\om}\}
  \end{displaymath}
\end{lemma}
\begin{proof}
  Consider first the case $s=0$. Then 
  \begin{displaymath}
    t^{\um}-s^{\um}=t^{\um}\geq \lambda^{\um/\om},
  \end{displaymath}
  since $t^{\om}\geq \lambda$. Assume then that $s>0$. By the mean value
  theorem, we have
  \begin{displaymath}
    t^{\um}-s^{\um}=(t^{\om})^{\um/\om}-(s^{\om})^{\um/\om}=
    \frac{\um}{\om}{\xi^{\um/\om-1}}(t^{\om}-s^{\om})
  \end{displaymath}
  for some $\xi\in (s^{\om},t^{\om})$. Since $\frac{\um}{\om}-1\leq 0$
  and $t^{\om}\leq M^{\om}$, we have
  \begin{displaymath}
    \frac{\um}{\om}{\xi^{\frac{\um}{\om}-1}}(t^{\om}-s^{\om})\geq 
    \frac{\um}{\om}M^{\um-\om}\lambda,
  \end{displaymath}
  which completes the proof.
\end{proof}

The following lemma provides the key estimate for showing that the
original sequence also converges to the limit found by applying the
compactness result.

\begin{lemma}\label{lem:oleinik}
  Let $U$ and $u_i$ be as in Theorem \ref{thm:compactness}, and
  let $0<t_1<t_2<T$.

  Fix a number $0<\delta<1$, and suppose that $u_{i,\delta}$ is the
  unique function which satisfies
  \begin{displaymath}
    \begin{cases}
      \partial_t u_{i,\delta}-\Delta u^{m_i}_i=0&\text{in }U_{t_1,t_2},\\
      u_{i,\delta}^{m_i}-u_{i}^{m_i}-\delta^{m_i}\in L^2(t_1,t_2;H^1_0(U)),&\\
      u_{i,\delta}(x,t_1)=u_i(x,t_1)+\delta, & x\in U.
    \end{cases}
  \end{displaymath}

  Then
  \begin{displaymath}
    \int_{U_{t_1,t_2}}(u_{i,\delta}-u_i)(u^{m_i}_{i,\delta}-u^{m_i}_i)\dif x\dif t
    \leq c(\delta+\delta^{m^-}),
  \end{displaymath}
  where
  \begin{displaymath}
    c=2^{m^++1}(M^{m^+}+M+1)\abs{U_{t_1,t_2}}
  \end{displaymath}
  and $M$ is the number appearing in \eqref{eq:assumptionM1}.
\end{lemma}

Before proceeding with the proof let us note that the technical reason
for introducing the exponents $\om_i$ and $\um_i$ is that we may write
\begin{displaymath}
  (u_{i,\delta}-u_i)(u^{m_i}_{i,\delta}-u^{m_i}_i)
  =(u_{i,\delta}^{\um_i}-u_i^{\um_i})(u^{\om_i}_{i,\delta}-u^{\om_i}_i)
\end{displaymath}
when applying this lemma. 

\begin{proof}
  The proof is an application of a test function due to Ole\u\i nik.
  The function $u_{i,\delta}^{m_i}-u_{i}^{m_i}-\delta^{m_i} $ has zero
  boundary values in Sobolev's sense, so the same is true for the
  function
  \begin{displaymath}
    \eta(x,t)=
    \begin{cases}
      \int_t^{t_2} u_{i,\delta}^{m_i}-u_{i}^{m_i}-\delta^{m_i}\dif s, & t_1<t<t_2,\\
      0, &t\geq t_2.
    \end{cases}
  \end{displaymath}
  We use $\eta$ as a test function in the equations satisfied by
  $u_{i,\delta}$ and $u_i$, and substract the results. This gives
  \begin{multline*}
    \int_{U_{t_1,t_2}}(u_{i,\delta}-u_i)(u_{i,\delta}^{m_i}-u_{i}^{m_i}
    -\delta^{m_i})\dif x\dif t\\
    \begin{aligned}
      +&\int_{U_{t_1,t_2}}\nabla(u_{i,\delta}^{m_i}-u_{i}^{m_i})
      \int_t^{t_2}\nabla(u_{i,\delta}^{m_i}-u_{i}^{m_i})\dif s\dif x\dif t\\
      =&\delta\int_{U}\int_{t_1}^{t_2} u_{i,\delta}^{m_i}-u_{i}^{m_i}\dif s\dif x.
    \end{aligned}
  \end{multline*}
  We move the term with $\delta$ to the right hand side, and integrate
  with respect to $t$ in the elliptic term. We get
  \begin{multline*}
    \int_{U_{t_1,t_2}}(u_{i,\delta}-u_i)(u^{m_i}_{i,\delta}-u^{m_i}_i)\dif
    x\dif t +\frac{1}{2}\int_U\left[\int_{t_1}^{t_2}\nabla
      (u_{i,\delta}^{m_i}-u_{i}^{m_i})\dif s\right]^2\dif x\\
    =\delta^{m_i}\int_{U_{t_1,t_2}}(u_{i,\delta}-u_i)\dif x\dif t
    +\delta\int_{U_{t_1,t_2}}(u^{m_i}_{i,\delta}-u_i^{m_i})\dif x\dif t
  \end{multline*}
  Now, both of the terms on the left are positive, so we have the
  freedom to take absolute values of the right hand side.  The claim
  then follows by discarding the second term on the left hand side,
  and estimating the integrals on the last line by using
  \eqref{eq:assumptionM1} and the fact that $u_{i,\delta}\leq M+1$, by
  the comparison principle. Indeed, we have
  \begin{displaymath}
    \scabs{\int_{U_{t_1,t_2}}(u_{i,\delta}-u_i)\dif x\dif t}\leq 2(M+1)\abs{U_{t_1,t_2}}
  \end{displaymath}
  and
  \begin{displaymath}
    \scabs{\int_{U_{t_1,t_2}}(u_{i,\delta}^{m_i}-u_i^{m_i})\dif x\dif t}\leq 
    2^{m^+}(M^{m^+}+1)\abs{U_{t_1,t_2}}.\qedhere
  \end{displaymath}
\end{proof}

The following lemma is the key step in the proof of Theorem
\ref{thm:compactness}.
\begin{lemma}\label{lem:sharp-conv}
  There exists a subsequence, still indexed by $i$, and a function $u$
  such that
  \begin{displaymath}
    u^{\om_i}_i\to u^{\om}\quad\text{as}\quad i\to\infty
  \end{displaymath}
  in measure and pointwise almost everywhere in $U_{t_1,t_2}$.
\end{lemma}

\begin{proof}
  Recall that $U$ is an open set such that $U\Subset\Omega$, and
  $0<t_1<t_2<T$.  We use the auxiliary functions $u_{i,\delta}$
  defined in Lem\-ma \ref{lem:oleinik}: for any $0<\delta<1$,
  $u_{i,\delta}$ is the weak solution in $U_{t_1,t_2}$ with boundary
  and initial values $u_i^{m_{i}}+\delta^{m_i}$ and
  $u(\cdot,t_1)+\delta$, respectively.  Note that $\delta\leq
  u_{i,\delta}\leq M+\delta$ by the comparison principle.

  Fix then a regular open set $V\Subset U$, and $t_1<s_1<s_2<t_2$.
  For indices such that $m_i\geq 1$, we have
   \begin{align*}
     \abs{\nabla u_{i,\delta}}=&\nabla (u_{i,\delta}^{m_i})^{1/m_{i}}
     =\abs{u_i^{1/m_i-1}\nabla u_{i,\delta}^{m_i}} \\ \leq & \delta^{1/m_i-1}
       \abs{\nabla u^{m_i}_{i,\delta}}\leq \delta^{1/m^+-1}
       \abs{\nabla u^{m_i}_{i,\delta}},
   \end{align*}
  and for indices such that $m_i<1$ we have
  \begin{align*}
    \abs{\nabla u_{i,\delta}}=&\abs{\nabla (u_{i,\delta}^{m_i})^{1/m_{i}}}\leq 
    \abs{u_i^{1/m_i-1}\nabla u_{i,\delta}^{m_i}}\\ \leq & (M+\delta)^{1/m_i-1}
    \abs{\nabla u^{m_i}_{i,\delta}}\leq (M+1)^{1/m^--1}
    \abs{\nabla u^{m_i}_{i,\delta}}.
  \end{align*}
  Thus Lemma \ref{lem:caccioppoli} implies that $(\nabla
  u_{i,\delta})$ is bounded in $L^2(V_{s_1,s_2})$. The fact that the
  sequence $(\partial_tu_{i,\delta})$ is bounded in
  $L^2(s_1,s_2;H^{-1}(V))$ follows easily from the equation satisfied
  by $u_i,\delta$ and Lemma \ref{lem:caccioppoli}. An application of
  \cite[Corollary 4]{Simon} shows that the sequence $(u_{i,\delta})$
  is compact in $L^2(V_{s_1,s_2})$. Thus for any fixed $\delta>0$, we
  may extract a subsequence which converges in $L^2(V_{s_1,s_2})$ and
  pointwise almost everywhere.

  The next step is a repeated application of the compactness
  established in the previous step. Let $\delta_j=1/j$,
  $j=1,2,3,\ldots,$ and for $j=1$ pick indices $i$ such that
  \begin{displaymath}
    u_{i,\delta_1}\to u_{\delta_1}
  \end{displaymath}
  for some function $u_{\delta_1}$, the convergence being in
  $L^2(V_{s_1,s_2})$ and pointwise almost everywhere. We proceed by
  picking a further subsequence so that also
  \begin{displaymath}
    u_{i,\delta_2}\to u_{\delta_2}
  \end{displaymath}
  in $L^2(V_{s_1,s_2})$ and pointwise almost everywhere.  We continue
  this way, and get a twodimensional table of indices $(i,j)$, where
  the values of $i$ in the diagonal positions have the property that
  \begin{displaymath}
    u_{i,\delta_j}\to u_{\delta_j} \text{ as }i\to \infty 
    \text{ for all }j,\text{ in }L^2(V_{s_1,s_2})\text{ and pointwise a.e.} 
  \end{displaymath}
  By the comparison principle, we have
  \begin{displaymath}
    u_{i,\delta_j}\geq u_{i,\delta_k}\quad\text{for}\quad k\geq j,
  \end{displaymath}
  and letting $i$ tend to infinity we get
  \begin{displaymath}
    u_{\delta_j}\geq u_{\delta_k}, \quad \text{whenever}\quad k\geq j.
  \end{displaymath}
  Thus we may define the function $u$ in the claim of the current
  lemma as the pointwise limit
  \begin{displaymath}
    u(x,t)=\lim_{j\to\infty} u_{\delta_j}(x,t).
  \end{displaymath}

  We have now found the limit function $u$, and the proof will be
  completed by showing that $u_i^{\om_i}$ converges to $u^{\om}$ in
  measure, first in $V_{s_1,s_2}$ and then in $U_{t_1,t_2}$ by an
  exhaustion argument. The first convergence is a consequence of
  Lemmas \ref{lem:diff-of-powers} and \ref{lem:oleinik}.  Obviously
  $u_{\delta_j}^{\om}\to u^{\om}$ pointwise almost everywhere as $j\to
  \infty$. Further, Lemma \ref{lem:conv-of-powers} implies that also
  $u_{i,\delta_j}^{\om_i}\to u_{\delta_j}^{\om}$ a.e. as $i\to
  \infty$.

  To show the convergence in measure, let $\lambda >0$ and
  $\varepsilon>0$. We have
  \begin{align*}
    \abs{\{\abs{u_i^{\om_i}-u^{\om}}\geq \lambda\}}\leq &
    \abs{\{\abs{u_i^{\om_i}-u^{\om_i}_{i,\delta_j}}\geq \lambda/3\}}\\
    +&\abs{\{\abs{u_{i,\delta_j}^{\om_i}-u^{\om}_{\delta_j}}\geq \lambda/3\}}
    +\abs{\{\abs{u_{\delta_j}^{\om}-u^{\om}}\geq \lambda/3\}}
  \end{align*}
  Since $u_{i,\delta_j}^{\um_i}-u_i^{\um_i}\geq
  c\min\{\lambda,\lambda^{\um_i/\om_i}\}$ by Lemma
  \ref{lem:diff-of-powers}, we have
  \begin{multline*}
    \abs{\{\abs{u_i^{\om_i}-u^{\om_i}_{i,\delta_j}}\geq \lambda/3\}}=
    \int_{\{\abs{u_i^{\om_i}-u^{\om_i}_{i,\delta_j}}\geq \lambda/3\}}1\dif x\dif t\\
    \begin{aligned}
      \leq & \frac{3}{\lambda}\int_{\{\abs{u_i^{\om_i}-u^{\om_i}_{i,\delta_j}}\geq \lambda/3\}}
      (u^{\om_i}_{i,\delta_j}-u^{\om_i}_{i})\dif x\dif t\\
      \leq & \frac{c}{\lambda\min\{\lambda,\lambda^{\um_i/\om_i}\}}
      \int_{V_{s_1,s_2}}(u_{i,\delta_j}^{\um_i}-u_{i}^{\um_i})(u^{\om_i}_{i,\delta_j}-u^{\om_i}_{i})\dif x\dif t\\
      \leq & \frac{c}{\lambda\min\{\lambda,\lambda^{\um_i/\om_i}\}}
      (\delta_j+\delta_j^{m^-}),
    \end{aligned}
  \end{multline*}
  where the last inequality follows from Lemma \ref{lem:oleinik} and
  the fact that
  \begin{displaymath}
    (u_{i,\delta}-u_i)(u^{m_i}_{i,\delta}-u^{m_i}_i)
    =(u_{i,\delta}^{\um_i}-u_i^{\um_i})(u^{\om_i}_{i,\delta}-u^{\om_i}_i).
  \end{displaymath} 
  Thus for each fixed $\lambda>0$, we may choose $\delta_j$ small
  enough, so that
  \begin{displaymath}
    \abs{\{\abs{u^{\om_i}_{i,\delta_j}-u_i^{\om_i}}\geq \lambda/3\}}\leq \varepsilon.
  \end{displaymath}
  Here it is crucial that this choice can be made independent of $i$.

  From the pointwise convergence established above, it follows that
  $u^{\om}_{\delta_j}\to u^{\om}$ in measure. Thus we may choose
  $\delta_j$ small enough, so that
  \begin{displaymath}
    \abs{\{\abs{u_{\delta_j}^{\om}-u^{\om}}\geq \lambda/3\}}\leq \varepsilon.
  \end{displaymath}
  
  Finally, we know that $ u_{i,\delta_j}^{\om_i}\to u_{\delta_j}^{\om}$
  pointwise almost everywhere and hence also in measure. Thus the
  above estimates imply that
  \begin{displaymath}
    \limsup_{i\to \infty}\abs{\{\abs{u_i^{\om_i}-u^{\om}}\geq \lambda\}}\leq
    \lim_{i\to\infty}\abs{\{\abs{u_{i,\delta_j}^{\om_i}-u^{\om}_{\delta_j}\}}
      \geq \lambda/3\}}
    +2\varepsilon=2\varepsilon.
  \end{displaymath}
  Since $\varepsilon$ was arbitrary, we get
  \begin{displaymath}
    \lim_{i\to\infty}\abs{\{\abs{u_i^{\om_i}-u^{\om}}\geq \lambda\}}=0
  \end{displaymath}
  for all $\lambda>0$, as desired. 
  We have proved that for cylinders $V_{s_1,s_2}$ such that
  $V_{s_1,s_2}\Subset U_{t_1,t_2}\Subset \Omega_T$, we may extract a
  subsequence such that
  \begin{displaymath}
    u_i^{\om_i}\to u^{\om} \quad\text{in measure in }V_{s_1,s_2} 
    \text{ as} \quad i\to\infty.
  \end{displaymath}
  Passing to a further subsequence, we have convergence pointwise
  almost everywhere. 

  To find a subsequence converging in the whole of $U_{t_1,t_2}$, we
  use a diagonalization argument. Exhaust $U$ by regular open sets
  $V^k$, and choose nested time intervals
  $(s_1^1,s_2^1)\Subset(s_1^2,s_2^2)\Subset\ldots $, $k=1,2,\ldots$,
  such that
  \begin{displaymath}
    U_{t_1,t_2}=\cup_{k=1}^\infty V_{s_1^k,s_2^k}^k.
  \end{displaymath}
  First, pick a subsequence such that $(u^{\om_i}_i)$ converges in
  measure in $V^1_{s_1^1,s_2^1}$ to $u^{\om}$. The procedure continues
  inductively, by the selection of a further subsequence that
  converges in measure in $V^{k+1}_{s_1^{k+1},s_2^{k+1}}$ to the
  function $u^{\om}$. Taking the $k$th index in the subsequence
  selected in the $k$th step yields a subsequence convergent in
  measure in $U_{t_1,t_2}$.
\end{proof}

The next step in the proof of Theorem \ref{thm:compactness} is to
prove the convergence of the other powers of $u$ by using the
convergence established in the previuos lemma.

\begin{lemma}\label{lem:sharp-to-others}
  Let $u_i$, $i=1,2,3,\ldots$, be such that $0\leq u_i\leq M<\infty$,
  and $u_i^{\om_i}\to u^{\om}$ in measure in $U_{t_1,t_2}$. Then
  $u_i\to u$ and $u_i^{m_i}\to u^m$ almost everywhere in $U_{t_1,t_2}$
  as $i\to \infty$.
\end{lemma}
\begin{proof}
  The bound $u_i\leq M$ and convergence in measure imply that
  $u_i^{\om_i}\to u^{\om}$ in $L^q(U_{t_1,t_2})$ for any finite
  $q$. Thus we obtain the desired convergences by two applications of
  Lemma \ref{lem:conv-of-powers}, first passing from the convergence
  $u_i^{\om_i}\to u^{\om}$ to the convergence $u_i\to u$, and then from 
  $u_i\to u$ to $u_i^{m_i}\to u^m$.
\end{proof}

With all of the preceding lemmas available, the proof of Theorem
\ref{thm:compactness} is now a relatively simple matter.
\begin{proof}[Proof of Theorem \ref{thm:compactness}]
  The pointwise convergences follow from Lemma
  \ref{lem:sharp-to-others} by an exhaustion argument, as in the proof
  of Lemma \ref{lem:sharp-conv}.  To show that $u$ is a local weak
  solution, fix a test function $\varphi\in C_0^\infty(\Omega_T)$. We
  use the bound $u_i\leq M$ and Lemma \ref{lem:caccioppoli} to get
  subsequences of $(u_i)$ and $(\nabla u^{m_i}_i)$, weakly convergent
  in $L^2(\spt \varphi)$. Due to pointwise convergences, we see that
  the weak limits must be $u$ and $\nabla u^m$, respectively. From the
  weak convergences, it follows that
  \begin{displaymath}
    \int_{\Omega_T}-u\frac{\partial \varphi}{\partial t}+\nabla u^m\cdot 
    \nabla \varphi\dif x\dif t=\lim_{i\to\infty}
    \left(\int_{\Omega_T}-u_i\frac{\partial \varphi}{\partial t}
      -\nabla u^{m_i}_i\cdot\nabla \varphi\dif x\dif t\right)=0.
  \end{displaymath}
  This holds for all test functions $\varphi$, so the proof is complete.
\end{proof}

\section{Stability of Dirichlet problems}
\label{sec:stab-dirichlet}

In this section we prove a stability result for Dirichlet boundary
value problems in bounded domains. We use Theorem
\ref{thm:compactness}, a local $L^\infty$ estimate (Proposition
\ref{prop:local-bound}), and a simple energy estimate (Lemma
\ref{lem:energy-est}). In this section, we take $\Omega$ to be
bounded.

\begin{definition}\label{def:pme-bvp}
  Let $u_0\in L^{m+1}(\Omega)$ and $g\in H^1(0,T;H^1(\Omega))$.  A
  positive function $u$ such that $u^m\in L^2(0,T;H^1(\Omega))$is a
  solution of the initial--boundary value problem
  \begin{equation}\label{eq:pme-bvp}
    \begin{cases}
      \partial_t u-\Delta u^{m}=0, & \text{in }\Omega_T,\\
      u^m=g, & \text{on }\partial \Omega\times [0,T],\\
      u(x,0)=u_0(x).
    \end{cases}
  \end{equation}
  if  $u^m-g\in L^2(0,T;H^1_0(\Omega))$, and 
  \begin{displaymath}
    \int_{\Omega_T}-u\frac{\partial\varphi}{\partial t}
    +\nabla u^m\cdot\nabla\varphi\dif x\dif
    t+\int_{\Omega}u(x,T)\varphi(x,T)\dif x=\int_\Omega
    u_0(x)\varphi(x,0)\dif x
  \end{displaymath}
  for all smooth test functions $\varphi$ which vanish on the lateral
  boundary of $\Omega_T$.
\end{definition}
For the existence and uniqueness of solutions in the above sense, see
\cite[Chapter 5]{VazquezBook}. By the usual approximation argument, we
may use test functions $\varphi\in L^2(0,T;H^1_0(\Omega))$.

Let $u_i$ be the solution to
\begin{equation}\label{eq:i-equations}
  \begin{cases}
    \partial_t u_i-\Delta u^{m_{i}}_i=0, & \text{in }\Omega_T,\\
    u^m_i=g, & \text{on }\partial \Omega\times [0,T],\\
    u_i(x,0)=u_0(x)
  \end{cases}
\end{equation}
in the sense of definition \ref{def:pme-bvp}. We will find a
subsequence of $(u_i)$ that converges in a suitable sense to a
function $u$, and show that $u$ is a solution of the limit problem,
i.e. satisfies
\begin{equation}\label{eq:limit-equation}
  \begin{cases}
    \partial_t u-\Delta u^{m}=0, & \text{in }\Omega_T,\\
    u^m=g, & \text{on }\partial \Omega\times [0,T],\\
    u(x,0)=u_0(x).
  \end{cases}
\end{equation}
More precisely, we have the following theorem.
\begin{theorem}\label{thm:stability}
  Let $m_i$, $i=1,2,3,\ldots$, be a sequence of exponents such that
  \begin{displaymath}
    m_i\to m > m_c=(n-2)_+/n\quad\text{as}\quad i\to \infty
  \end{displaymath}
  Let $u_i$ be the solutions to \eqref{eq:i-equations} with fixed
  initial and boundary values $g$ and $u_0$, where 
  \begin{displaymath}
    g\in H^1(0,T; H^1(\Omega)),\quad
    \frac{\partial g}{\partial t}\in
    L^{1+1/m^-}(\Omega_T),\quad\text{and}\quad u_0\in L^{m^++1}(\Omega).
  \end{displaymath}
  Finally, let $u$ be the solution to \eqref{eq:limit-equation} with
  the boundary and initial values $g$ and $u_0$, respectively.
  
  Then 
  \begin{enumerate} 
  \item $u_i\to u$ in $L^{q}(\Omega_T)$ for all $1\leq q<1+m$.
  \item $u_i^{m_i}\to u^m$ in $L^s(\Omega_T)$ for all $1\leq s<2\kappa$, where
    \begin{displaymath}
      \kappa=1+\frac{1}{m}+\frac{1}{mn}.
    \end{displaymath}
  \item $\nabla u_i^{m_i}\to \nabla u^m$ weakly in $L^2(\Omega_T)$.
  \end{enumerate}
\end{theorem}

We need an energy estimate for establishing that the limit function
attains the right boundary values in Sobolev's sense, and for
verifying the local boundedness assumption in Theorem
\ref{thm:compactness}. To derive it, we use the equation satisfied by
the mollified solution $u^\ast$:
\begin{equation}
  \label{eq:reg-bvp}
  \int_{\Omega_T}\varphi\frac{\partial u^\ast}{\partial t}
  +\nabla (u^m)^\ast\cdot\nabla\varphi \dif x\dif t=  
  \int_{\Omega}u_0(x)\left(\frac{1}{\sigma}\int_{0}^T\varphi e^{-s/\sigma}
    \dif s\right)\dif x
\end{equation}
This is required to hold for all test functions $\varphi\in L^2(0,T;
H^1_0(\Omega))$.  The equation \eqref{eq:reg-bvp} follows from
\eqref{eq:pme-bvp} by straightforward manipulations involving a change
of variables and Fubini's theorem.

\begin{lemma}\label{lem:energy-est}
  Let $u$ be the weak solution with boundary values $g$ and initial
  values $u_0$. Then
  \begin{multline}
    \esssup_{0<t<T}\int_{\Omega}u^{m+1}(x,t)\dif x
    +\int_{\Omega_T}\abs{\nabla u^m}^2\dif x\dif t\\
    \leq c\left(\int_{\Omega_T} \abs{\nabla g}^2
      +\scabs{\frac{\partial g}{\partial t}}^{1+1/m} \dif x\dif t 
      +\int_{\Omega}u_0(x)^{m+1}\dif x\right).\label{eq:energy-est}
  \end{multline}
\end{lemma}
\begin{proof}
  We test the regularized equation \eqref{eq:reg-bvp} with
  $\varphi=u^m-g$. This yields
  \begin{multline*}
    \int_{\Omega_T}\frac{\partial u^\ast}{\partial t}(u^m-g)\dif x\dif t
    +\int_{\Omega_T}\nabla (u^m)^\ast\cdot\nabla(u^m-g)\dif x\dif t\\
    =\int_{\Omega}u_0(x)\left(\frac{1}{\sigma}\int_{0}^T(u^m-g)e^{-s/\sigma}
      \dif s\right)\dif x.
  \end{multline*}
  We rearrange this to get
  \begin{multline}\label{eq:energyest-pf1}
    \int_{\Omega_T}\frac{\partial u^\ast}{\partial t} u^m\dif x\dif t
    +\int_{\Omega_T}\nabla (u^m)^\ast\cdot\nabla u^m\dif x\dif t\\
    \begin{aligned}
      = & \int_{\Omega_T}\frac{\partial u^\ast}{\partial t} g\dif x\dif t
      +\int_{\Omega_T}\nabla (u^m)^\ast\cdot \nabla g\dif x\dif t\\
      &+\int_{\Omega}u_0(x)\left(\frac{1}{\sigma}\int_{0}^T(u^m-g)e^{-s/\sigma}
      \dif s\right)\dif x.
    \end{aligned}
  \end{multline}
  
  For the first term on the left, we get 
  \begin{align*}
    \int_{\Omega_T}\frac{\partial u^\ast}{\partial t} u^m\dif x\dif t
    =&\int_{\Omega_T}\frac{\partial u^\ast}{\partial t} (u^\ast)^m\dif
    x\dif t
    +\int_{\Omega_T}\frac{\partial u^\ast}{\partial t} (u^m-(u^\ast)^m)\dif x\dif t\\
    \geq &\int_{\Omega_T}\frac{\partial u^\ast}{\partial t} (u^\ast)^m\dif x\dif t
    =\int_{\Omega}\frac{u^\ast(x,T)^{m+1}}{m+1}\dif x\\
    \to & \int_{\Omega}\frac{u(x,T)^{m+1}}{m+1}\dif x
  \end{align*}
  as $\sigma\to 0$. Here we used \eqref{eq:naumann-timederiv}
  to get the inequality.  For the second, it suffices to note that
  \begin{displaymath}
    \int_{\Omega_T}\nabla (u^m)^\ast\cdot\nabla u^m\dif x\dif t\to 
    \int_{\Omega_T} \abs{\nabla u^m}^2\dif x\dif t
  \end{displaymath}
  as $\sigma\to 0$.
  
  The limits of the left hand side terms in \eqref{eq:energyest-pf1}
  are positive, so we are free to take absolute values on the right
  after passing to the limit $\sigma\to 0$. In the first term on the
  left, we integrate by parts before taking the limit:
  \begin{displaymath}
    \int_{\Omega_T}\frac{\partial u^\ast}{\partial t}g\dif x\dif t
    =-\int_{\Omega_T}u^\ast\frac{\partial g}{\partial t}\dif x\dif t
    \to -\int_{\Omega_T}u\frac{\partial g}{\partial t}\dif x\dif t.
  \end{displaymath}
  We proceed by taking absolute values, applying Young's inequality,
  and taking the supremum over $t$.  We get
  \begin{displaymath}
    \scabs{\int_{\Omega_T}u\frac{\partial g}{\partial t}\dif x\dif t}\leq 
    \varepsilon \esssup_{0<t<T}\int_{\Omega}u^{1+m}(x,t)\dif x
    +c_\varepsilon\int_{\Omega_T}
    \scabs{\frac{\partial g}{\partial t}}^{1+1/m}\dif x\dif t.
  \end{displaymath}
  
  The limit of the second term on the right of
  \eqref{eq:energyest-pf1} is
  \begin{displaymath}
    \int_{\Omega_T}\nabla u^m\cdot \nabla g\dif x\dif t.
  \end{displaymath}
  Here we simply take absolute values and use Young's inequality to get
  \begin{displaymath}
    \scabs{\int_{\Omega_T}\nabla u^m\cdot \nabla g\dif x\dif t}\leq 
    \varepsilon \int_{\Omega_T}\abs{\nabla u^m}^2\dif x\dif t
    +c\int_{\Omega_T}\abs{\nabla g}^2\dif x\dif t.
  \end{displaymath}
  
  For the third term on the right of \eqref{eq:energyest-pf1}, taking
  the limit yields
  \begin{displaymath}
    \int_{\Omega}u_0^{m+1}(x)\dif x-\int_{\Omega}u_0(x)g(x,0)\dif x;
  \end{displaymath}
  The first term needs no further estimations, and the second is
  negative so we may discard it.

  We have arrived at 
  \begin{multline}
    \label{eq:energyest-pf2} 
        \int_{\Omega} \frac{u^{m+1}(x,T)}{m+1}\dif x
    +\int_{\Omega_T}\abs{\nabla u^m}^2\dif x\dif t\\
    \begin{aligned}
    \leq & \varepsilon \int_{\Omega_T}\abs{\nabla u^m}^2\dif x\dif t 
      +\varepsilon\esssup_{0<t<T}\int_\Omega u(x,t)^{1+m}\dif x\\
      &+c\int_{\Omega}u_0^{m+1}(x)\dif x 
      +c\left(\int_{\Omega_T}\abs{\nabla g}^2
        +\scabs{\frac{\partial g}{\partial t}}^{1+1/m}\dif x\dif t\right).
    \end{aligned}
  \end{multline}
  To finish the proof, replace $T$ in the above proof by a number
  $0<\tau<T$ such that
  \begin{displaymath}
    \int_{\Omega} u(x,\tau)^{m+1}\dif x\geq 
    \frac{1}{2}\esssup_{0<t<T}\int_{\Omega} u(x,t)^{m+1}\dif x. 
  \end{displaymath}
  This leads to an estimate for
  \begin{displaymath}
    \esssup_{0<t<T}\int_{\Omega} u(x,t)^{m+1}\dif x
    +\int_{\Omega_\tau}\abs{\nabla u^m}^2\dif x\dif t
  \end{displaymath}
  in terms of the right hand side of \eqref{eq:energyest-pf2}. The claim
  follows by choosing $\varepsilon$ sufficiently small and then
  absorbing the matching terms to the left hand side.
\end{proof}

An analysis of the above proof shows that the constant may be taken to
depend only on $m^+$ and $m^-$, as $m$ varies over the interval
$[m^-,m^+]$. Thus, in view of \eqref{eq:energy-est} and the
assumptions on $g$ and $u_0$ in Theorem \ref{thm:stability}, we see
that the sequence $(\nabla u_i^{m_i})$ is bounded in $L^2(\Omega_T)$,
and by the Sobolev embedding, $(u_i^{m_i})$ is bounded in
$L^{2}(\Omega_T)$.

We use the following estimate to verify the local boundedness
assumption in Theorem \ref{thm:compactness}. See \cite[Proposition B.5.1]{DGV}.

\begin{proposition}\label{prop:local-bound}
  Let $u$ be a local weak solution, and suppose that
  $B_\rho\times[t_0-\rho^2,t_0]\Subset \Omega_T $. Then
  \begin{displaymath}
    \esssup_{B_{\rho/2}\times [t_0-\rho^2/2,t_0]} u\leq 
    c\left(\vint_{B_\rho\times[t_0-\rho^2,t_0]}u\dif x\dif t\right)^{2/\lambda}
    +1,
  \end{displaymath}
  where
  \begin{displaymath}
    \lambda=n(m-1)+2>0.
  \end{displaymath}
\end{proposition}
As is discussed in \cite{DGV}, the constants in estimates of this type
are stable as $m$ decreases or increases to one, but blow up as $m\to
m_c$ or $m\to \infty$. Hence we may again assume that the constant is
independent of $i$, as $m_i$ varies in the interval $[m^-,m^+]$.  Note
also that we have written this estimate over standard parabolic
cylinders $B_\rho\times(t_0-\rho^2,t_0)$, for otherwise the estimate
would not be stable.

The proof of our stability result is now a straightforward matter.
\begin{proof}[Proof of Theorem \ref{thm:stability}]
  We combine Lemmas \ref{lem:sobolev} and \ref{lem:energy-est} with
  Proposition \ref{prop:local-bound} to verify the local boundedness
  assumption \eqref{eq:assumptionM1} in Theorem
  \ref{thm:compactness}. Thus the theorem yields a subsequence and a
  function $\widetilde{u}$ such that $u_i\to \widetilde{u}$ and
  $u^{m_i}_i\to \widetilde{u}^m$, pointwise almost everywhere in
  $\Omega_T$.  We may also assume that $\nabla u_{i}^{m_i}\to \nabla
  \widetilde{u}^m$ weakly in $L^2(\Omega_T)$, by Lemma
  \ref{lem:energy-est} and the pointwise convergence.

  To pass to the limit in the equations, we need a bound on $u_i$. 
  We have
  \begin{align*}
    \int_{\Omega_T}u_i^{1+m^-}\dif x\dif t \leq &
    c\esssup_{0<t<T}\int_{\Omega}u^{1+m^-}(x,t)\dif x\\
     \leq & c\left(\esssup_{0<t<T}\int_{\Omega}u_i^{1+m_i}(x,t)\dif x+1\right).
  \end{align*}
  This together with Lemma \ref{lem:energy-est} implies that the
  sequence $(u_i)$ is bounded in $L^{1+m^-}(\Omega_T)$. By
  reflexivity, we may assume that $u_i\to \widetilde{u}$ weakly.

  By applying the weak convergences, we see that $\widetilde{u}$ must
  be a solution of the limit equation with the right boundary and
  initial values. First, it is clear that
  \begin{displaymath}
    \widetilde{u}^m-g\in L^2(0,T;H^1_0(\Omega)),
  \end{displaymath}
  since $L^2(0,T;H^1_0(\Omega))$ is weakly closed, by virtue of being
  a closed subspace of $L^2(0,T;H^1(\Omega))$. Further, we have
  \begin{align*}
    \int_{\Omega_T}-\widetilde{u}\frac{\partial \varphi}{\partial t}
    +\nabla \widetilde{u}^m\cdot \nabla \varphi\dif x\dif t
    =&\lim_{i\to \infty}\int_{\Omega_T}-u_i\frac{\partial \varphi}{\partial t}
    +\nabla u_i^{m_i}\cdot \nabla \varphi\dif x\dif t\\
    =&\int_\Omega u_0(x)\varphi(x,0)\dif x
  \end{align*}
  for all smooth test functions $\varphi$ vanishing on the lateral
  boundary and at $t=T$ by the weak convergences; by uniqueness of
  weak solutions, this means that $\widetilde{u}=u$.

  Convergence in measure and a bound in $L^p$ imply convergence in
  $L^r$ for any $r<p$. For the pointwise convergent
  subsequences, the claims 
  \begin{displaymath}
    u_i\to u \quad \text{in}\quad L^q(\Omega_T), \quad 1\leq q<1+m, 
  \end{displaymath}
  and 
  \begin{displaymath}
    u_i^{m_i}\to u^m\quad \text{in}\quad L^s(\Omega_T), \quad 1\leq s<2\kappa 
  \end{displaymath}
  follow from this. The corresponding convergence of the original
  sequence then follows from the fact that any subsequence that
  converges, must converge to the same limit, by the uniqueness of
  weak solutions. This also holds for weak convergence, whence we get
  the remaining claim about the weak convergence of the gradients.
\end{proof}

\begin{remark}
  Generalizing Theorem \ref{thm:stability} to, e.g., Neumann boundary
  conditions is a matter of proving a suitable counterpart of the
  energy estimate of Lemma \ref{lem:energy-est}. Indeed, Proposition
  \ref{prop:local-bound} is a purely local estimate, independent of
  any boundary conditions. We leave the details to the interested
  reader.
\end{remark}

\section{Stability of Cauchy problems}
\label{sec:stab-cauchy}
In this section, we prove stability of Cauchy problems on the whole
space $\R^n$. As for Dirichlet problems, Theorem \ref{thm:compactness}
is the key tool. Other results we use are the $L^1$ contraction
property, and the $L^1-L^\infty$ regularizing effect

\begin{definition}\label{def:cauchy}
  Let $\mu$ be a positive Borel measure on $\R^n$ such that
  \begin{displaymath}
    \mu(\R^n)<\infty.
  \end{displaymath}
  A positive function $u:\R^n\times(0,\infty)\to [0,\infty)$ is a
  solution to the initial value problem
  \begin{displaymath}
    \begin{cases}
      u_t-\Delta u^m=0 & \text{in }\R^n\times(0,\infty),\\
      u(x,0)=\mu
    \end{cases}
  \end{displaymath}
  if the following hold.

  \begin{enumerate}
  \item $u\in L^\infty(0,\infty; L^1(\R^n))$, $u\in
    L^{\infty}(\R^n\times(\tau,\infty))$ for all $\tau>0$, and $u^m\in
    L^1(S\times (0,T))$ for all compact subsets $S$ of $\R^n$ and
    finite $T$.

  \item $u$ is a local weak solution in $\R^n\times(0,\infty)$ in the
    sense of Definition \ref{def:local-weak}.
  \item For all test functions $\varphi\in C^\infty_0(\R^n\times
    [0,\infty))$, it holds that
    \begin{equation}\label{eq:cauchy-weak}
      \int_{\R^n\times[0,\infty)} -u\frac{\partial \varphi}{\partial t}
      -u^m\Delta\varphi
      \dif x\dif t=\int_{\R^n}\varphi(x,0)\dif \mu
    \end{equation}
  \end{enumerate}
\end{definition}

Solutions with the above properties exist and are unique. Indeed, by
approximating $\mu$ with nice initial data, we can construct solutions
such that the estimates of Theorem \ref{thm:cauchy-estimates} below
hold, and these estimates give the properties of $u$ in Definition
\ref{def:cauchy}. See \cite{DaskalopoulosKenig, VazquezBook} for the
details.  For the uniqueness, we note that \eqref{eq:cauchy-weak} and
the integrability of $u^m$ up to the initial time imply that
\begin{equation}\label{eq:initialtrace}
  \int_{\R^n}u(x,\tau)\eta(x)\dif x\to \int_{\R^n}\eta(x)\dif \mu
  \quad\text{as}\quad\tau\to 0.
\end{equation}
for all $\eta\in C_0^\infty(\R^n)$; thus we may appeal to the
uniqueness results in \cite{HerreroPierreUniq, PierreUniq}. See also
\cite{DaskalopoulosKenig} for an account of this uniqueness theory.

The Barenblatt solutions \eqref{eq:barenblatt-pme} and
\eqref{eq:barenblatt-fde} furnish examples of a solution to the Cauchy
problem in the sense of Definition \ref{def:cauchy}. One usually
normalizes $\B_m$ by choosing the constant $C$ in
\eqref{eq:barenblatt-pme}, \eqref{eq:barenblatt-fde} so that
\begin{displaymath}
  \int_{\Omega}\B_m(x,t)\dif x=1
\end{displaymath}
for all $t>0$. With this normalization, $\B_m$ is the unique solution
to the Cauchy problem with initial trace $\mu$ given by Dirac's delta
at the origin, as a straightforward computation shows.

For Cauchy problems, we have the following stability result.
\begin{theorem}\label{thm:cauchy-stability}
  Let $m_i$, $i=1,2,3,\ldots$, be a sequence of exponents such that
  \begin{displaymath}
    m_i\to m > m_c=(n-2)_+/n\quad\text{as}\quad i\to \infty,
  \end{displaymath}
  and let $u_i$ be the solutions to 
  \begin{equation}
    \label{eq:i-cauchy}
    \begin{cases}
      \partial_t u_i-\Delta u_i^{m_i}=0 & \text{in }\R^n\times(0,\infty),\\
      u_i(x,0)=\mu
    \end{cases}
  \end{equation}
  with fixed initial trace $\mu$. Further, let $u$ be the solution to
  the limit problem
  \begin{equation}
    \label{eq:cauchy-stab}
    \begin{cases}
      \partial_t u-\Delta u^{m}=0 & \text{in }\R^n\times(0,\infty),\\
      u_i(x,0)=\mu
    \end{cases}
  \end{equation}
  with the same initial trace $\mu$.
  
  Then for all compact sets $S$ in $\R^n$ and all finite $T$, we have
  \begin{enumerate}
  \item $u_i\to u$ in $L^q(S_T)$ for all $1\leq q<m+2/n$, 
  \item $u_i^{m_i}\to u^m$ in $L^s(S_T)$ for all $1\leq s< 1+2/mn$,
  \item $\nabla u_i^{m_i}\to \nabla u^m$ weakly in
    $L^2_{\mathrm{loc}}(\R^n\times(0,\infty))$.
  \end{enumerate}  
\end{theorem}

The following theorem provides the necessary estimates for our
stability result. The admissible range of the integrability exponent
$q$ in \eqref{eq:cauchy3} is sharp for the Barenblatt solution. This
can be checked by a simple computation.
\begin{theorem}\label{thm:cauchy-estimates}
  Let $u\geq 0$ be a solution to the Cauchy problem with initial data
  $\mu$ such that
  \begin{displaymath}
    \norm{\mu}=\mu(\R^n)<\infty.
  \end{displaymath}
  Then the following estimates hold.
  
  For all $t>0$, we have
  \begin{equation}\label{eq:cauchy1}
    \norm{u(\cdot,t)}_{L^1(\R^n)}\leq \norm{\mu}.
  \end{equation}

  For every $t>0$, we have
  \begin{equation}\label{eq:cauchy2}
    u(x,t)\leq c\norm{\mu}^{2/\lambda}t^{-n/\lambda}.
  \end{equation}
  where
  \begin{displaymath}
    \lambda= n(m-1)+2.
  \end{displaymath}

  The function $u^m$ belongs to $L^q(S_T)$ for all compact sets $S$ in
  $\R^n$ and all finite $T$, for 
  \begin{displaymath}
    1\leq q<1+\frac{2}{mn}.
  \end{displaymath}
  We also have the estimate
  \begin{equation}\label{eq:cauchy3}
    \int_{S_T}u^{mq}\dif x\dif t\leq c\norm{\mu}^{\frac{2}{\lambda}(mq-1)+1}
    T^{-\frac{n}{\lambda}(mq-1)+1}.
  \end{equation}

\end{theorem}
\begin{proof}
  The inequalities \eqref{eq:cauchy1} and \eqref{eq:cauchy2} are
  standard estimates for the Cauchy problem, see
  \cite{DaskalopoulosKenig, VazquezBook2, VazquezBook}. The inequality
  \eqref{eq:cauchy3} is a slight refinement of the well-known fact
  that $u^m$ is integrable up to the initial time locally in space,
  and follows from the first two. For the reader's convenience, we
  present the computation here.  By applying \eqref{eq:cauchy2} in the
  first inequality and \eqref{eq:cauchy1} in the second, we have
  \begin{align*}
    \int_{S_T} u^{mq}\dif x\dif t\leq & c\norm{\mu}^{\frac{2}{\lambda}(mq-1)}
    \int_{S_T}u(x,t)t^{-\frac{n}{\lambda}(mq-1)}\dif x\dif t\\
    \leq & c\norm{\mu}^{\frac{2}{\lambda}(mq-1)+1}\int_{0}^T 
    t^{-\frac{n}{\lambda}(mq-1)}\dif t.
  \end{align*}
  We may evaluate the integral with respect to time and obtain
  \eqref{eq:cauchy3} if
  \begin{displaymath}
    -\frac{n}{\lambda}(mq-1)>-1,
  \end{displaymath}
  which is equivalent with 
  \begin{displaymath}
    q<1+\frac{2}{mn}.\qedhere
  \end{displaymath}
\end{proof}

The sharp constants in \eqref{eq:cauchy2}, as $m$ varies, are given in
\cite[p. 26]{VazquezBook2}. The situation is similar to that of
Proposition \ref{prop:local-bound}: the constants are stable as $m$
either increases or decreases to one, but blow up as $m\to m_c$ or
$m\to \infty$. Thus we are again free to assume that the constants are
independent of $i$ as $m_i$ varies in the interval $[m^-,m^+]$.

\begin{proof}[Proof of Theorem \ref{thm:cauchy-stability}.]
  We use \eqref{eq:cauchy2} to conclude that the sequence $(u_i)$ is
  locally bounded in $\R^n\times (0,\infty)$. Thus Theorem
  \ref{thm:compactness} gives us pointwise convergent subsequences of
  $(u_i)$ and $(u_i^{m_i})$, with limits $\widetilde{u}$ and
  $\widetilde{u}^m$, respectively.

  Convergence in measure and a bound in $L^p$ imply convergence in
  $L^r$ for any $r<p$. From this, the claims
  \begin{displaymath}
    u_i\to \widetilde{u} \quad \text{in}\quad L^q(S_T), \quad 1\leq  q<m+2/n, 
  \end{displaymath}
  and 
  \begin{displaymath}
    u_i^{m_i}\to \widetilde{u}^m\quad \text{in}\quad L^s(S_T), \quad 1\leq  s<1+\frac{2}{mn}
  \end{displaymath}
  follow easily. We use these convergences to conclude that
  \begin{multline*}
    \int_{\R^n\times[0,\infty)}-\widetilde{u}\frac{\partial \varphi}{\partial t}
    -\widetilde{u}^m\Delta\varphi\dif x\dif t\\
    = \lim_{i\to \infty} \left(\int_{\R^n\times[0,\infty)}-u_i
      \frac{\partial \varphi}{\partial t}
      -u_i^{m_i}\Delta\varphi\dif x\dif t\right)
    =\int_{\R^n}\varphi(x,0)\dif \mu.
  \end{multline*}
  for any test function $\varphi\in C_0^\infty(\R^n\times
  [0,\infty))$.  The uniqueness of solutions to the Cauchy problem
  \eqref{eq:cauchy-stab} now implies that $\widetilde{u}=u$.  The
  convergences for the original sequence then follow from the fact
  that all convergent subsequences converge to the same limit, by
  uniqueness.
\end{proof}

\begin{remark}
  Generalizing Theorem \ref{thm:cauchy-stability} to Cauchy problems
  with growing initial data, described in e.g. \cite[Chapters 2 and
  3]{DaskalopoulosKenig} or \cite[Chapter 13]{VazquezBook}, offers no
  additional difficulties.  Indeed, counterparts for all the estimates
  in Theorem \ref{thm:cauchy-estimates} above are available; see for
  instance \cite[Theorem 13.1]{VazquezBook} for appropriate
  replacements of \eqref{eq:cauchy1} and \eqref{eq:cauchy2}. An
  estimate similar to \eqref{eq:cauchy3} then follows by repeating the
  above computation.  With these estimates in hand, the proof of
  Theorem \ref{thm:cauchy-stability} requires virtually no
  modifications. We leave the details to the interested reader.
\end{remark}

\section{Stability of Dirichlet problems revisited}

\label{sec:alt-stab}

In this section, we present an alternative proof of Theorem
\ref{thm:stability} in the case $m_i\geq 1$. The advantage of this
argument is that we avoid using the local $L^\infty$ estimate,
Proposition \ref{prop:local-bound}.

For the reader's convenience, we give the statement of the theorem
before proceeding with the proof, althought this is essentially the
same as Theorem \ref{thm:stability} with the additional assumption
$m_i\geq 1$.
\begin{theorem}
  \label{thm:alt-dirichlet-stab}
  Let $m_i$, $i=1,2,3,\ldots$, be a sequence of exponents such that
  and
  \begin{displaymath}
    m_i\geq 1 \quad\text{and}\quad  
    m_i\to m  \quad\text{as}\quad i\to \infty.
  \end{displaymath}
  Let $u_i$, $i=1,2,3,\ldots$, be the solutions to
  \begin{equation}\label{eq:alt-i-equations}
    \begin{cases}
      \partial_t u_i-\Delta u^{m_{i}}_i=0, & \text{in }\Omega_T,\\
      u^m_i=g, & \text{on }\partial \Omega\times [0,T],\\
      u_i(x,0)=u_0
    \end{cases}
  \end{equation}
  with fixed initial and boundary values $g$ and $u_0$, where
  \begin{displaymath}
    g\in H^1(0,T; H^1(\Omega)),\quad
    \text{and}\quad u_0\in L^{m^++1}(\Omega).
  \end{displaymath}
  Finally, let $u$ be the solution to 
  \begin{equation}\label{eq:alt-limit-equation}
    \begin{cases}
      \partial_t u-\Delta u^{m}=0, & \text{in }\Omega_T,\\
      u^m=g, & \text{on }\partial \Omega\times [0,T],\\
      u(x,0)=u_0.
    \end{cases}
  \end{equation} with
  the same boundary and initial values $g$ and $u_0$.
  
  Then 
  \begin{enumerate} 
  \item $u_i\to u$ in $L^{q}(\Omega_T)$ for all $1\leq q<1+m$.
  \item $u_i^{m_i}\to u^m$ in $L^s(\Omega_T)$ for all $1\leq s<2\kappa$, where
    \begin{displaymath}
      \kappa=1+\frac{1}{m}+\frac{1}{mn}.
    \end{displaymath}
  \item $\nabla u_i^{m_i}\to \nabla u^m$ weakly in $L^2(\Omega_T)$.
  \end{enumerate}
\end{theorem}

The following elementary inequality is needed in the proof.
\begin{lemma}\label{lem:alt-powers}
  For positive $t$, we have
  \begin{displaymath}
    \abs{t^\alpha-t^\beta}\leq c_\varepsilon(1+t^{\alpha+\varepsilon}+t^{\beta+\varepsilon})
    \abs{\alpha-\beta}
  \end{displaymath}
\end{lemma}
\begin{proof}
  This follows by an application of the mean value theorem to the
  function $x \mapsto t^{x}$. We have
  \begin{displaymath}
    \frac{\dif }{\dif x }t^x=  t^x \log t,
  \end{displaymath}
  which is estimated by 
  \begin{displaymath}
     t^x \log t\leq c_\varepsilon(1+t^{\alpha+\varepsilon}+t^{\beta+\varepsilon})
  \end{displaymath}
  for $x$ in the interval $(\alpha,\beta)$.
\end{proof}

Stability is a consequence of the following theorem. It provides a
quantitative estimate of the difference of two solutions.
\begin{theorem}\label{thm:modulus}
  Let the exponents $m_i$ and $m$, and the functions $u_i$ and $u$ be
  as in Theorem \ref{thm:alt-dirichlet-stab}. Then
  \begin{displaymath}
    \norm{u-u_i}_{L^{1+m}(\Omega_T)}\leq c\abs{m-m_i}^{1/m},
  \end{displaymath}
  for indices $i$ large enough, where the constant depends on the
  norms of the boundary and initial values $g$ and $u_0$ appearing in
  Lemma \ref{lem:energy-est}.
\end{theorem}

\begin{proof}
  We aim at using the standard inequality
  \begin{equation}\label{eq:monotonicity}
    c\abs{a-b}^{1+m}\leq  (a^{m}-b^{m})(a-b)
  \end{equation}
  in combination of an application of Ole\u\i nik's test function. 
  The function 
  \begin{displaymath}
    \eta(x,t)=
    \begin{cases}
      \int_t^T u^{m}-u^{m_i}_i\dif s, & 0<t<T,\\
      0, & \text{otherwise},
    \end{cases}
  \end{displaymath}
  has zero boundary values in Sobolev's sense on the lateral boundary.
  We test the equations satisfied by $u$ and $u_i$, and substract
  the results. This leads to 
  \begin{equation}\label{eq:modulus-proof1}
    \begin{aligned}
      \int_{\Omega_T}(u-u_i)(u^{m}-u^{m_i}_i)\dif x\dif t
      =& -\int_{\Omega_T}\nabla(u^{m}-u^{m_i}_i)\cdot
      \int_t^T\nabla(u^{m}-u^{m_i}_i)\dif s\dif x\dif t\\
      =&-\frac{1}{2}\int_{\Omega}
      \left[\int_0^T\nabla(u^{m}-u^{m_i}_i)\dif s\right]^2\dif x  
      \leq  0,
    \end{aligned}
  \end{equation}
  where we integrated with respect to $t$ to get the last line.
  Thus we have
  \begin{align*}
    c\int_{\Omega_T}\abs{u-u_i}^{1+m}\dif x\dif t  
    \leq &\int_{\Omega_T}(u-u_i)(u^{m}-u_i^{m})\dif x\dif t\\
      = & \int_{\Omega_T}(u-u_i)(u^{m}-u_i^{m_i})\dif x\dif t\\
       &+\int_{\Omega_T}(u-u_i)(u_i^{m_i}-u_i^{m})\dif x\dif t\\
      \leq & \int_{\Omega_T}(u-u_i)(u_i^{m_i}-u_i^{m})\dif x\dif t
    \end{align*}
    by \eqref{eq:monotonicity} and \eqref{eq:modulus-proof1}.

    We proceed by an application of H\"older's inequality, and get
    \begin{align*}
      \int_{\Omega_T}(u-u_i)(u_i^{m_i}-u_i^{m})\dif x\dif t
      \leq & \left(\int_{\Omega_T}\abs{u-u_i}^{1+m}\dif x\dif
        t\right)^{1/(1+m)}\\
      \times &\left(\int_{\Omega_T}\abs{u_i^{m_i}-u_i^{m}}^{(m+1)/m}
        \dif x\dif t\right)^{m/(m+1)}.
    \end{align*}
    Then we apply Lemma \ref{lem:alt-powers} inside the second
    integral, and get
    \begin{displaymath}
      \abs{u_i^{m_i}-u_i^{m}}\leq c(1+u_i^{m_i+\varepsilon}+u_i^{m+\varepsilon})
      \abs{m-m_i}.
    \end{displaymath}
    The choice of $\varepsilon$ will be made later. Thus
    \begin{multline*}
      \left(\int_{\Omega_T}\abs{u_i^{m_i}-u_i^{m}}^{(m+1)/m}\dif
        x\dif t\right)^{m/(m+1)}\\
      \leq c\left(\int_{\Omega_T}
        (1+u_i^{m_i+\varepsilon}+u_i^{m+\varepsilon})^{(m+1)/m}
        \dif x \dif t\right)^{m/(m+1)}\abs{m-m_i}.
    \end{multline*}
    We put all the estimates together, and end up with
    \begin{multline*}
      \norm{u-u_i}_{L^{m+1}(\Omega_T)} \\ \leq
      c\left(\int_{\Omega_T}
        (1+u_i^{m_i+\varepsilon}+u_i^{m+\varepsilon})^{(m+1)/m}
        \dif x \dif t\right)^{1/(m+1)}\abs{m-m_i}^{1/m}
    \end{multline*}

    To finish, we need a uniform (in $i$) $L^1$ bound for the the
    functions
    \begin{equation}\label{eq:modulus-proof12}
      u_i^{m_i(1+\varepsilon/m_i)(1+1/m)}\quad\text{and}\quad 
      u_i^{m_i(m/m_i+\varepsilon/m_i)(1+1/m)},
    \end{equation}
    at least for large $i$. We combine the Sobolev embedding (Lemma
    \ref{lem:sobolev}) and the energy estimate (Lemma
    \ref{lem:energy-est}), and get that $u_i^{2\kappa_i m_i}$,
    $i=1,2,\ldots$, is bounded in $L^1(\Omega)$, where $\kappa_i$ is
    given by
    \begin{displaymath}
      \kappa_i=1+\frac{1}{n}+\frac{1}{m_in}.
    \end{displaymath}
    The desired conclusion follows by proving that we may choose
    $\varepsilon$ small enough, so that the exponents in
    \eqref{eq:modulus-proof12} to are less than $2\kappa_im_i$ for
    large $i$. Indeed, we may choose $i$ large enough and
    $\varepsilon$ small enough, so that $1+\varepsilon/m_i$ and
    $m/m_i+\varepsilon/m_i$ are arbitrarily close to one, and so that
    $\kappa_i$ is arbitrarily close to $\kappa$. Hence to satisfy the
    conditions
    \begin{displaymath}
      m_i(1+\frac{\varepsilon}{m_i})(1+\frac{1}{m})
      <2\kappa_im_i\quad\text{and}\quad
      m_i(\frac{m}{m_i}+\frac{\varepsilon}{m_i})(1+\frac{1}{m})<2\kappa_im_i
    \end{displaymath}
    for large $i$, it suffices to require that
    \begin{displaymath}
      1+\frac{1}{m}<2\kappa.
    \end{displaymath}
    A computation shows that this holds if
    \begin{displaymath}
      m>\frac{n-2}{n+2},
    \end{displaymath}
    which is quaranteed by our assumption $m\geq 1$.  Thus above
    arguments give the desired bound
    \begin{displaymath}
      \norm{u-u_i}_{L^{1+m}(\Omega_T)}\leq c\abs{m-m_i}^{1/m}
    \end{displaymath}
    for $i$ large enough.
\end{proof}

\begin{remark}
  The reason why the above proof does not work well if $m<1$ is the
  use of \eqref{eq:monotonicity}. A similar inequality is of course
  available for $m<1$, but the exponents that come up in the course of
  the proof blow up as $m\to 1$.
\end{remark}

\begin{proof}[Proof of Theorem \ref{thm:alt-dirichlet-stab}]
  Theorem \ref{thm:modulus} implies that the sequence $(u_i)$
  converges in $L^{1+m}(\Omega_T)$ to $u$, the solution of the limit
  problem. For a subsequence we get pointwise a.e. convergence, and an
  application of Lemma~\ref{lem:conv-of-powers} then yields the
  pointwise convergence $u_i^{m_i}\to u^m$. After the pointwise
  convergences are available, the rest of the convergence claims are
  established as in the proof of Theorem~\ref{thm:stability}.
\end{proof}

\subsubsection*{Acknowledgement} The author would like to thank
professors Boris Adreianov and  Philippe Lauren\c{c}ot for pointing
out relevant references.

\def\cprime{$'$}

\end{document}